\documentclass[12pt]{article}

\usepackage[a4paper]{geometry}
\usepackage{amssymb}
\renewcommand{\geq}{\geqslant}
\renewcommand{\leq}{\leqslant}

\usepackage{amsthm,amsmath,amssymb}
\usepackage{graphicx}
\usepackage[colorlinks=true,citecolor=black,linkcolor=black,urlcolor=blue]{hyperref}

\usepackage{verbatim}

\newcommand{\arxiv}[1]{\href{http://arxiv.org/abs/#1}{\texttt{arXiv:#1}}}

\newtheorem{thm}{Theorem}[section]
\newtheorem{prop}[thm]{Proposition}
\newtheorem{lem}[thm]{Lemma}
\newtheorem{cor}[thm]{Corollary}

\theoremstyle{definition}

\newtheorem{ex}[thm]{Example}

\theoremstyle{remark}
\newtheorem{rem}[thm]{Remark}

\newcommand{\N}{\mathbb{N}}

\newcommand{\R}{\mathbb{R}}
\newcommand{\K}{\mathbb{K}}

\newcommand{\A}{\mathcal{A}}

\newcommand{\C}{\mathcal{C}}

\renewcommand{\S}{\mathcal{S}}
\newcommand{\G}{\mathcal{G}}
\newcommand{\E}{\mathcal{E}}

\renewcommand{\P}{\mathbb{P}}

\newcommand{\PP}{\mathbf{P}}
\newcommand{\NP}{\mathbf{NP}}

\newcommand{\len}{\mathop{\rm len}}
\newcommand{\type}{\mathop{\rm type}}

\newcommand{\Graph}{\mathop{\rm Graph}}
\newcommand{\NCSym}{\rm NCSym}
\newcommand{\NCQSym}{\rm NCQSym}

\title{\bf Plurigraph coloring and scheduling problems}

\author{John Machacek\\
\small Department of Mathematics\\[-0.8ex]
\small Michigan State University\\[-0.8ex] 
\small East Lansing, MI 48824, USA\\
\small\tt machace5@math.msu.edu\\
}

\date{\small Mathematics Subject Classifications: 05C15, 05E05}

\begin{document}

\maketitle

\begin{abstract}
We define a new type of vertex coloring which generalizes vertex coloring in graphs, hypergraphs, and simplicial complexes.
This coloring also generalizes oriented coloring, acyclic coloring, and star coloring.
There is an associated symmetric function in noncommuting variables for which we give a deletion-contraction formula.
In the case of graphs this symmetric function in noncommuting variables agrees with the chromatic symmetric function in noncommuting variables of Gebhard and Sagan.
Our vertex coloring is a special case of the scheduling problems defined by Breuer and Klivans.
We show how the deletion-contraction law can be applied to certain scheduling problems.
Also, we show that the chromatic symmetric function determines the degree sequence of uniform hypertrees, but there exist pairs of $3$-uniform hypertrees which are not isomorphic yet have the same chromatic symmetric function.

  \bigskip\noindent \textbf{Keywords:} graph coloring, deletion-contraction, chromatic symmetric function, scheduling problem, plurigraph
\end{abstract}

\section{Introduction}
We define a generalization of vertex coloring which has many variations of graph coloring as special cases.
It generalizes the usual notion of proper coloring in graphs and hypergraphs.
We also show that a certain type of coloring in simplicial complexes corresponds to coloring uniform hypergraphs, and hence this version of coloring in simplicial complexes is included in our theory.
Furthermore, we find that oriented coloring, acyclic coloring, and star coloring show up as special cases of our generalized coloring.
Associated to our generalization of vertex coloring we have a symmetric function in noncommuting variables which generalizes the chromatic symmetric function in noncommuting variables defined by Gebhard and Sagan in~\cite{GS}.
The vertex coloring we define can be thought of as a special class of the scheduling problems defined by Breuer and Klivans, and our symmetric function in noncommuting variables is an instance of the scheduling quasisymmetric function in noncommuting variables from~\cite{BK}.
Our theory allows for a unified way to study a chromatic symmetric function for all the types of coloring listed above.

In this section we will give the basic definitions of symmetric functions in noncommuting variables, introduce scheduling problems, and review vertex coloring in graphs and hypergraphs.
Section~\ref{sec:gen} is where we define \emph{plurigraphs} and our new notion of coloring.
We also prove a deletion-contraction law.
The connections with the theory of scheduling problems is given in Section~\ref{sec:sched}.
In Section~\ref{sec:dist} we discuss what information the chromatic symmetric function can tell us about plurigraphs and uniform hypertrees.
We will show how oriented coloring, acyclic coloring, and star coloring show up as instances of plurigraph coloring in Section~\ref{sec:other}.

An extended abstract which included some results in this paper appeared in the Proceedings of the 28th International Conference on Formal Power Series and Algebraic Combinatorics (FPSAC) 2016, Vancouver, Canada~\cite{fpsac}.

\subsection{NCSym and NCQSym}
We let $\P = \{1,2,\dots\}$ denote the set of positive integers and for any $n \in \P$ we let $[n] = \{1,2,\dots,n\}$.
A \emph{partition} of $[n]$ is $\pi = B_1 / B_2/ \cdots / B_{\ell}$ where $\biguplus_{i=1}^l B_i = [n]$.
When writing partitions we often suppress notation by simply writing $12/3$ in place of $\{1,2\}/\{3\}$.
Here we call each $B_i$ a \emph{block} of the partition $\pi$ and the order of the blocks is irrelevant.
For example $12/3$ and $3/12$ denote the same partition of $[3]$.
Take noncommuting variables $\{y_1, y_2, \dots\}$ and a partition $\pi$ of $[n]$ for some $n \in \P$, then the \emph{monomial nc-symmetric function} $m_{\pi}$ is defined by
\begin{equation}
m_{\pi} := \sum_{i_1,i_2,\dots,i_n} y_{i_1}y_{i_2} \cdots y_{i_n}
\label{eq:monomial}
\end{equation}
where the sum is over all sequences $(i_1,i_2, \dots, i_n) \in \P^n$ satisfying the condition that $i_j = i_k$  if and only if $j$ and $k$ are in the same block of the partition $\pi$.
For example,
\[m_{12/3} = y_1y_1y_2 + y_2y_2y_1 + y_1y_1y_3 + y_3y_3y_1 + \cdots\]
is a monomial nc-symmetric function.
The \emph{powersum nc-symmetric function} $p_{\pi}$ is defined by
\begin{equation}
p_{\pi} := \sum_{\sigma \geq \pi} m_{\sigma}
\label{eq:powersum}
\end{equation}
where $\sigma \geq \pi$ is taken in the lattice of partitions of $[n]$ partially ordered by refinement.
For example,
\[p_{12/3} = m_{12/3} + m_{123}\]
is a powersum nc-symmetric function.

We denote the lattice of partitions of $[n]$ by $\Pi_n$.
We now define $\NCSym$ the algebra of \emph{nc-symmetric functions} to be the $\K$-space generated by either the basis of monomial nc-symmetric functions $\{m_{\pi}: \pi \in \Pi_n, n \in \P\}$ or the basis of powersum nc-symmetric functions $\{p_{\pi} : \pi \in \Pi_n, n \in \P\}$.
The algebra $\NCSym$ can be defined over any field $\K$.
We will assume $\K$ is a field of characteristic zero throughout since the assumption of characteristic zero will be needed for some results.

A \emph{composition} of $[n]$ is $\Phi = (B_1, B_2, \dots, B_{\ell})$ where $\biguplus_{i=1}^{\ell} B_i = [n]$.
When writing composition we often suppress notation in the same manner as for a partition by simply writing $(12,3)$ in place of $(\{1,2\},\{3\})$.
We again call each $B_i$ a \emph{block} of the composition $\Phi$.
The order of the blocks of a composition of $[n]$ is relevant.
For example, $(12,3)$ and $(3,12)$ denote different compositions of $[3]$.
Taking a composition $\Phi = (B_1,B_2,\dots,B_{\ell})$ of $[n]$ for some $n \in \P$ the \emph{monomial nc-quasisymmetric function} is defined by
$$M_{\Phi} := \sum_{i_1,i_2,\dots,i_n} y_{i_1}y_{i_2} \cdots y_{i_n}$$
where the sum is over all sequences $(i_1,i_2, \dots, i_n) \in \P^n$ satisfying the condition that for $j,k \in [n]$ where $j \in B_p$ and $k \in B_q$ with $p \leq q$ we have that 
\begin{itemize}
\item $i_j \leq i_k$
\item $i_j = i_k$ if and only if $p=q$.
\end{itemize}
We let $\Delta_n$ denote the collection of all compositions of $[n]$ and define the algebra of \emph{nc-quasisymmetric functions}, which we denote by $\NCQSym,$ to be the $\K$-space generated by the basis of monomial nc-quasisymmetric functions $\{M_{\Phi} : \Phi \in \Delta_n, n \in \P\}$.
As an example we have
\begin{align*}
M_{(12,3)} &= y_1y_1y_2 + y_1y_1y_3 + y_2y_2y_3 + \cdots\\
M_{(3,12)} &= y_2y_2y_1 + y_3y_3y_1 + y_3y_3y_2 + \cdots
\end{align*}
as elements of $\NCQSym$.
We note that $\NCSym$ is contained in $\NCQSym$ as a proper subset.
Given $\pi \in \Pi_n$ we have
$$m_{\pi} = \sum_{\Phi} M_{\Phi}$$
where the sum is over all compositions $\Phi$ of $[n]$ which have the same blocks as $\pi$.
One can check $m_{12/3} = M_{(12,3)} + M_{(3,12)}$ using the previous examples in this section.

Given a monomial $y_{i_1}y_{i_2} \cdots y_{i_n}$ and a sequence $(r_0,r_1,\dots,r_k) \in \P^{k+1}$ with $k < n$ we define the \emph{induction} of the monomial with respect to the sequence by
$$y_{i_1}y_{i_2} \cdots y_{i_n} \uparrow^{(r_0,r_1,\dots,r_k)} := y_{i_1}y_{i_2} \cdots y_{i_{n-k}}^{1+r_k}\cdots y_{i_{n-1}}^{1+r_1}y_{i_n}^{1+r_0}.$$
Extending this operation linearly we get induction of any element of $\NCQSym$, and hence any element of $\NCSym$ as well.
This generalizes induction as defined in~\cite{GS} where what we denote by $\uparrow^{(1)}$ is used.
We can define induction on compositions in a way which is compatible with induction on $\NCQSym$.
Given a composition $\Phi \in \Delta_n$ , $r \in \P$, and $t \in \N$ define $\Phi\uparrow^{s,t}$ to be the composition of $[n+s]$ obtained from $\Phi$ by first replacing $n-j$ with $n-j+s$ for $0 \leq j < t$ and then placing $n-t+1, n-t+2, \dots, n-t+s$ in the same block as $n-t$.
For $(r_0, r_1, \dots, r_k) \in \P^{k+1}$ with $k < n$ we define $\Phi\uparrow^{(r_0, r_1, \dots, r_k)} := \Phi\uparrow^{(r_0, r_1, \dots, r_k),0}$ where we have the recursion $\Phi\uparrow^{(r_0, r_1, \dots, r_j),t} = \left(\Phi\uparrow^{r_0,t}\right)\uparrow^{(r_1, r_2, \dots, r_j),t+r_0+1}$ with $\Phi\uparrow^{(),t} = \Phi$.

\begin{lem}
If $\Phi \in \Delta_n$ and $(r_0, r_1, \dots, r_k) \in \P^{k+1}$ with $k < n$, then $M_{\Phi} \uparrow^{(r_0, r_1, \dots, r_k)} = M_{\Phi \uparrow^{(r_0, r_1, \dots, r_k)}}$.
\label{lem:induction}
\end{lem}
The lemma follows from the definition of the induction operation.
Induction can be applied to partitions in the same way as compositions but without regarding the ordering of the blocks.
Lemma~\ref{lem:induction} then implies that
\begin{align*}
m_{\pi}\uparrow^{(r_0, r_1, \dots, r_k)} &= m_{\pi \uparrow^{(r_0, r_1, \dots, r_k)}} &p_{\pi}\uparrow^{(r_0, r_1, \dots, r_k)} &= p_{\pi\uparrow^{(r_0, r_1, \dots, r_k)}}
\end{align*}
whenever $\pi \in \Pi_n$ and $k < n$.
We now demonstrate the induction operation with an example.

\begin{ex}
We take $(1,2) \in \Delta_2$ and $(2,1) \in \P^2$.
First let us consider induction on the composition $(1,2)$.
\begin{align*}
(1,2)\uparrow^{(2,1)} &= (1,2)\uparrow^{(2,1),0}\\
&= (1,234)\uparrow^{(1),3}\\
&= (12,345)
\end{align*}
Next we consider induction on the nc-monomial quasisymmetric function and see that is compatible with induction on the composition.
\begin{align*}
M_{(1,2)} &= y_1y_2 + y_1y_3 + y_2y_3 + \cdots\\
M_{(1,2)}\uparrow^{(2,1)} &= y_1^2y_2^3 + y_1^2y_3^3 + y_2^2y_3^3 + \cdots\\
M_{(1,2)\uparrow^{(2,1)}} &= M_{(12,345)} = y_1^2y_2^3 + y_1^2y_3^3 + y_2^2y_3^3 + \cdots
\end{align*}
\end{ex}

\subsection{Scheduling problems}
As defined in \cite{BK} a \emph{scheduling problem} on $n$ elements is a boolean formula $S$ over the atomic formulas $(x_i \leq x_j)$ for $i, j \in [n]$.
We are interested in solutions to a scheduling problem where each $x_i$ takes a value in $\P$.
A function $f:[n] \to \P$ is a solution to the scheduling problem $S$ if when $x_i = f(i)$ the boolean formula $S$ is true.
We then get \emph{the scheduling nc-quasisymmetric function} $\S_S$ defined by
$$\S_S := \sum_f \prod_{i=1}^n y_{f(i)}$$
where the sum is taken over all solutions $f$ to the scheduling problem $S$.
Given $\Phi = (B_1,B_2,\dots,B_{\ell})$ a composition of $[n]$ we can view $\Phi$ as a map $\Phi:[n] \to [\ell]$ by $f(i) = j$ if $i \in B_j$.
We say a set composition $\Phi$ solves $S$ if its corresponding map does.
In this way we see that $\S_S$ is indeed an element of $\NCQSym$ and can be expressed in the monomial basis as
$$\S_S = \sum_{\Phi} M_{\Phi}$$
where the sum is over set compositions $\Phi$ of $[n]$ which solve $S$.

\subsection{Coloring in graphs, hypergraphs, and simplicial complexes}
For us a graph is will mean a finite undirected graph with loops and multiple edges allowed.
We will write a graph $G$ as a pair $G = (V, E)$ where where $V$ is a finite set and $E$ is finite multiset of unordered pairs of (not necessarily distinct) elements of $V$.
We call elements of $V$ vertices and elements of $E$ edges.
When $|V| = n$ we will usually assume without stating that $V = [n]$.
From identifying $V$ with $[n]$ we obtain an ordering of the vertices.
Given vertices $u,v \in V$, the edge between $u$ and $v$ is written $uv \in E$ where $uv = vu$.
A map $f:V \to \P$ is called a proper coloring of $G$ if it produces no monochromatic edge.
That is $f$ is a proper coloring if for all $uv \in E$ we have that $f(u) \ne f(v)$.

A hypergraph $H$ is a pair $H = (V,E)$ where $E$ a collection of nonempty subsets of $V$.
We call the elements of $V$ vertices and elements of of $E$ hyperedges.
If for each $e \in E$ we have that $|e| = s$, then we call $H$ an $s$-uniform hypergraph.
A map $f:V \to \P$ is a proper coloring of $H$ if it produces no monochromatic hyperedge.

An abstract simplicial complex $\Gamma$ on a vertex set $V$ is a collection of subsets of $V$ such that for all $v \in V$ we have $\{v\} \in \Gamma$ and if $A \in \Gamma$ then $B \in \Gamma$ for any $B \subseteq A$.
Elements of $\Gamma$ are called faces and faces which are maximal with respect to inclusion are called facets.
We call $A \in \Gamma$ an $s$-simplex if $|A| = s+1$.
Given a positive integer $s,$ a map $f:V \to \P$ is an $s$-simplicial coloring of $\Gamma$ if it produces no monochromatic $s$-simplex.
Coloring in graphs and hypergraphs is classical, but this notion of coloring in simplicial complexes is more recent and defined in~\cite{dmn:vcsc}.

We will now show that coloring in simplicial complexes can be thought of as coloring in uniform hypergraphs and conversely.
Given a simplicial complex $\Gamma$ and a nonnegative integer $s$ we get an $(s+1)$-uniform hypergraph $H^{(s)}(\Gamma) = (V(\Gamma), E^{(s)}(\Gamma))$ with edge set defined by
$$E^{(s)}(\Gamma) := \{A \in \Gamma: |A| = s+1\}.$$

\begin{lem}
A map $f$ is an $s$-simplicial coloring of a simplicial complex $\Gamma$ if and only if $f$ is a proper coloring of $H^{(s)}(\Gamma)$.
\label{lem:equiv2}
\end{lem}

Next we show that coloring in a uniform hypergraph can be thought of as an instance of coloring in a simplicial complex.
Given any $H = (V,E)$ we get a simplicial complex $\Gamma(H)$ on the vertex set $V$ defined by
$$\Gamma(H) = \{A : A \subseteq e \in E\} \cup \{\{v\} : v \in V\}.$$

\begin{lem}
A map $f$ is a proper coloring of an $(s+1)$-uniform hypergraph $H$ if and only if $f$ is an $s$-simplicial coloring of $\Gamma(H)$.
\label{lem:equiv1}
\end{lem}

Note a simplicial complex is determined completely by its facets, and we can rephrase the property of being an $s$-simplicial coloring in terms of facets.
For a simplicial complex $\Gamma$ with vertex set $V$ a map $f:V \to \P$ is an $s$-simplicial coloring if and only if each facet of $\Gamma$ contains at most $s$ vertices of a given color.
If $H$ is an $(s+1)$-uniform hypergraph, then $\Gamma(H)$ is a simplicial complex with facets given by the hyperedges of $H$ along with possibly a some isolated vertices.
In particular, the simplicial complex $\Gamma(H)$ can be obtained from $H$ in linear time. 
 
We define the decision problem $(k,s)$-$\textsc{simplicial colorable}$ which takes as input a simplicial complex $\Gamma$ and outputs true if and only if $\Gamma$ can be $s$-simplicial colored using at most $k$ colors.
Similarly we define the decision problem $(k,s)$-$\textsc{colorable}$ which takes as input a $s$-uniform hypergraph $H$ and outputs true if and only if $H$ can be properly colored using at most $k$ colors.
Thus Lemma~\ref{lem:equiv1} says by considering $H \mapsto \Gamma(H)$ we have a polynomial time reduction from $(k,s+1)$-$\textsc{colorable}$ to $(k,s)$-$\textsc{simplicial colorable}$.
We use the fact that $(k,s)$-$\textsc{colorable}$ is $\NP$-complete unless $k=1$, $s=1$, or $(k,s)=(2,2)$~\cite{Lovasz}.

\begin{prop}
We have $(k,s)$-$\textsc{simplicial colorable} \in \PP$ for $k = 1$ or $(k,s) = (2,1)$ 
and for all other pairs $(k,s)$ we have that $(k,s)$-$\textsc{simplicial colorable}$ is $\NP$-complete.
\label{prop:NP}
\end{prop}
\begin{proof}
We certainly have the $(k,s)$-$\textsc{simplicial colorable} \in \NP$ for any $(k,s)$.
Note the if $s = 1$ we are considering graph coloring and if $k=1$ we simply need to determine the dimension of the simplicial complex.
The $\NP$-hardness for $k \ne 1$ and $(k,s) \ne (2,1)$ follows from the polynomial time reduction from $(k,s+1)$-$\textsc{colorable}$ by $H \mapsto \Gamma(H)$.
\end{proof}

\section{Coloring and plurigraphs}
\label{sec:gen}

We define a \emph{plurigraph} $\G$ to be a pair $\G = (V, \E)$ where $V$ is a finite set and $\E$ is a multiset of nonempty graphs with vertex set $V$.
This means elements of $\E$ are of the form $(V,E)$ with $E \ne \emptyset$.
An element $(V,E) \in \E$ is called a \emph{pluriedge}, and $(V,E) \in \E$ is a \emph{pluriloop} if $E$ consists of only loops.
Observe that a pluriedge is just a graph.
We use this terminology and the notation $\E$ to emphasize that in the theory of plurigraph coloring the elements $(V,E) \in \E$ play a role analogous to the role edges play in the classical theory of graph coloring.

If $\G = (V, \E)$ is a plurigraph, then a coloring of $\G$ is a map $f:V \to \P$. 
The coloring $f:V \to \P$ is a \emph{proper coloring} of $\G$ if for each $(V,E) \in \E$ there exists an edge $e \in E$ which is not monochromatic.
Observe that no proper coloring of $\G$ exists if $\G$ contains a pluriloop.
Also observe that $f: V \to \P$ is a proper coloring of $\G$ if and only if for each $(V,E) \in \E$ there exists a connected component of $(V,E)$ which is not monochromatic.
So, plurigraph coloring depends on the components of the pluriedges.

\begin{ex}
Given a graph $G = (V, E)$ we get a plurigraph $\G_G := (V, \E_G)$ where $\E_G = \{(V, \{e\}) : e \in E\}$.
Here proper colorings of $\G_G$ exactly correspond to proper colorings of $G$.
For a hypergraph $H = (V,E)$ we get the plurigraph $\G_H := (V, \E_H)$ where $\E_H = \{(V,E_e) : e \in E\}$ and $E_e = \{uv: u,v \in e\}$.
The proper colorings of $\G_H$ are in correspondence with proper colorings of $H$.
\label{ex:graph}
\end{ex}

\begin{rem}
For a hypergraph $H = (V,E)$ the encoding of $H$ as the plurigraph $\G_H$ in Example~\ref{ex:graph} is canonical,  but it is not the most efficient.
One could define $\G'_{H} = (V, \E'_H)$ where  $\E'_H = \{(V,E'_e) : e \in E\}$ and $E'_e$ consists of the edges of any tree on the vertices of the hyperedge $e$.
We would still have the property that there is a one-to-one correspondence between proper colorings of $H$ and $\G'_H$.
In general $\G'_H$ will use many fewer edges than $\G_H$.
We will use $\G_H$ because it does not require of choice of tree for each hyperedge.
\label{rem:efficient}
\end{rem}

We have shown in Lemma~\ref{lem:equiv2} that $s$-simplicial coloring is equivalent to proper coloring $(s+1)$-uniform hypergraphs.
Therefore considering Example~\ref{ex:graph}, coloring in plurigraphs encompasses coloring in graphs, hypergraphs, and simplicial complexes.
In Section~\ref{sec:other} we will show that coloring in plurigraphs also encompasses other types on coloring like oriented coloring and acyclic coloring.

For a plurigraph $\G$ we define the \emph{chromatic nc-symmetric function} of $\G$ by
$$Y_{\G} := \sum_f \prod_{i=1}^n y_{f(i)}$$
where the sum is over all proper colorings $f$ of $\G$.
It is readily verified that $Y_{\G}$ is in fact an element of $\NCSym$.
By allowing the variables to commute we obtain the symmetric function $X_{\G}$ which we call the \emph{chromatic symmetric function} of $\G$.
We also obtain the \emph{chromatic polynomial} of $\G$, which we denote $\chi_{\G}$, by letting $\chi_{\G}(k)$ by the specialization of $Y_{\G}$ with $y_i = 1$ for $1 \leq i \leq k$ and $y_i = 0$ for $i > k$.
Such a specialization always gives arise to a polynomial~\cite[Proposition 7.8.3]{EC2}.
Here $\chi_{\G}(k)$ counts the number of proper coloring of $\G$ using only colors from $[k]$.

For a plurigraph $\G = (V, \E)$ and $(V,E) \in \E$, \emph{deletion} of the pluriedge $(V,E)$ is denoted $\G \setminus (V,E)$ and defined by
$$\G \setminus (V,E) := (V, \E')$$
where $\E' = \E \setminus \{(V,E)\}$.
For any graphs $(V,E_1)$ and $(V,E_2)$ on the same vertex set we define 
$$(V,E_1)/(V,E_2) := (V/(V,E_2), E_1/(V,E_2))$$
where $V/(V,E_2)$ and $E_1/(V,E_2)$ are obtained by identifying the vertices $u$ and $v$ whenever $uv \in E_2$.
For a plurigraph $\G = (V, \E)$ and $(V,E) \in \E$ \emph{contraction} by the pluriedge $(V,E)$ is denoted by $\G / (V,E)$ and defined by  
$$\G / (V,E) = (V/(V,E), \E'')$$
where $\E'' = \{G/(V,E) : G \in \E'\}$ and again $\E' = \E \setminus (V,E)$.
Observe that these definitions agree with the usual notion of deletion and contraction in a graph $G$ if we consider the plurigraph $\G_G$ from Example~\ref{ex:graph}.
Also, note that for a pluriloop deletion and contraction are equivalent.
Lastly, notice that both deletion and contraction always decrease the number of pluriedges by exactly 1.

Let $\G = (V, \E)$ be a plurigraph and let $(V,E) \in \E$ be a fixed pluriedge.
For any two disjoint subsets $A, B \subseteq V$ we say $A < B$ if $a < b$ for all $a \in A$ and $b \in B$.
We call the pluriedge $(V,E)$ \emph{contraction-ready} if the blocks of the partition of $V$ given by the connected components of the graph $(V,E)$ can be ordered to obtain the composition $\Phi = (B_1,B_2, \dots, B_{\ell})$ of $V$ such that $B_{i_1} > B_{i_2}$ for any $i_1 < i_2$, and there is some $k \geq 0$ such that $B_i$ a singleton if and only if $i > k$.
In this case we call $\Phi$ a contraction-ready composition.
Notice that any pluriedge can be made contraction-ready by some relabeling of the vertices.
Here if $|V| = n$ we identify $V$ with $[n]$, and relabeling the vertices amounts to acting by some permutation $\delta \in S_n$.
The permutation $\delta$ also acts on a monomial of degree $n$ by
$$\delta(y_{i_1} y_{i_2} \cdots y_{i_n}) = y_{i_{\delta^{-1}(1)}} y_{i_{\delta^{-1}(2)}} \cdots y_{i_{\delta^{-1}(n)}}$$
and can be extended linearly to act on $Y_{\G}$.
In this case we have $\delta(Y_{\G}) = Y_{\delta(\G)}$.
For the corresponding relabeling result for graphs see~\cite[Proposition 3.3]{GS}.
Thus we can assume that any plurigraph $\G$ has a contraction-ready pluriedge since we can always obtain such a pluriedge by relabeling.

If $(V,E)$ is a contraction-ready pluriedge then the contraction operation is compatible with our ordering of vertices.
When the contraction-ready composition is $\Phi = (B_1,B_2, \dots, B_{\ell})$ the $\ell$ blocks of $\Phi$ will correspond to the vertex set of the contraction $\G / (V,E)$, and the vertex set will be $[\ell]$ where the vertices in $B_i$ all get identified to a single vertex denoted by $\ell - i + 1$.
To demonstrate this consider the following example.

\begin{ex}
Consider the plurigraph $\G = ([4], \{([4], \{13,24\}), ([4],\{12,34\})\})$ and let $G = ([4],\{12,34\})$.
Then $\G \setminus G = ([4], \{([4], \{13,24\})\})$ and $\G / G = ([2], \{([2],\{12,12\})\})$.
Here the contraction-ready composition is $(34,12)$.
The vertices $3$ and $4$ are identified and denote by $2$ while the vertices $1$ and $2$ are identified and denoted by $1$ in $\G / G$.
An example of a proper coloring of $\G$ is given by $f:[4] \to \P$ by $f(1)=f(2)=f(3)=1$ and $f(4) = 2$.
The plurigraphs $\G$, $\G \setminus G$, and $\G / G$ are shown visually in Figure~\ref{fig:example}.
\label{ex:contraction}
\end{ex}

\begin{figure}
\centering
\includegraphics[scale=0.8]{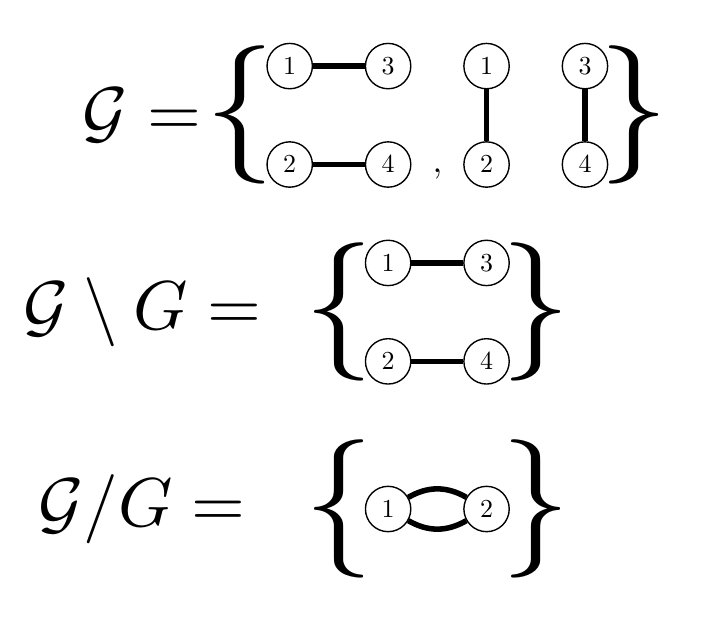}
\caption{A visual representation of $\G$, $\G \setminus G$, and $\G / G$ from Example~\ref{ex:contraction}}
\label{fig:example}
\end{figure}

We are now ready to state and prove the deletion-contraction formula for the chromatic nc-symmetric function of a plurigraph.

\begin{thm}
If $\G = (V, \E)$ and $(V,E) \in \E$ is a contraction-ready pluriedge, then
$$Y_{\G} = Y_{\G \setminus (V,E)} - Y_{\G / (V,E)} \uparrow^{(r_1,r_2,\dots,r_k)}$$
where $\Phi = (B_1,B_2, \dots, B_{\ell})$ is the contraction-ready composition of $V$ given by connected components of $(V,E)$ and $r_i = |B_i| - 1$.
\label{thm:delcongraph}
\end{thm}
\begin{proof}
Given $f:V \to \P$ observe that $f$ being monochromatic on all edges $e \in E$ is equivalent to being $f$ monochromatic on each connected component of $(V,E)$.
Now we have
$$Y_{\G \setminus (V,E)} =\sum_f \prod_{i=1}^n y_{f(i)} $$
where the sum is over proper colorings of $\G \setminus (V,E)$.
Any proper coloring of $\G \setminus (V,E)$ is of one of the following two types.
\begin{itemize}
\item A proper coloring of $\G$.
\item A proper coloring of $\G \setminus (V,E)$ which is monochromatic on each connected component of $(V,E)$.
\end{itemize}
Then by the definition on contraction, induction, and the fact the $(V,E)$ was contraction-ready we have that
$$Y_{\G / (V,E)} \uparrow^{(r_1,r_2,\dots,r_k)} = \sum_f \prod_{i=1}^n y_{f(i)}$$
where the sum is over $f:V \to \P$ such that $f$ is a proper coloring of $\G \setminus (V,E)$ which is monochromatic on each component of $(V,E)$.
The theorem readily follows.
\end{proof}

\begin{rem}
We make note here of a few special cases that are covered by Theorem~\ref{thm:delcongraph} and its proof, but that we think are worth explicitly mentioning.
First, if $(V,E) \in \E$ occurs more than once in the multiset $\E$, then $Y_{\G} = Y_{\G \setminus (V,E)}$ and $Y_{\G / (V,E)} = 0$ since $\G / (V,E)$ will contain a pluriloop.
Hence we see the theorem is true and the proof works the same as there are no proper colorings of $\G \setminus (V,E)$ which are monochromatic on each connected component of $(V,E)$.
Second if $(V,E) \in \E$ is a pluriloop, then $Y_{\G} = 0$ and $Y_{\G \setminus (V,E)} = Y_{\G / (V,E)}$.
Here induction need not be performed since all blocks in the contraction-ready composition will be singletons.
Also there are no colorings of either type mentioned in the proof.
So, again we see the theorem and proof are valid in the special case.
\end{rem}

Our deletion-contraction law in Theorem~\ref{thm:delcongraph} generalizes the deletion-contraction law in~\cite[Proposition 3.5]{GS}, and in a similar way we can use our deletion-contract law to give a powersum expansion~\cite[Theorem 3.6]{GS} of the chromatic nc-symmetric function $Y_{\G}$.
Given a plurigraph $\G = (V, \E)$ and $A \subseteq \E$ with $A = \{(V,E_i) : 1 \leq i \leq k\}$ we define $\pi(A)$ to be the partition of $V$ into the connected components of the graph $(V,\bigcup_{i=1}^k E_i)$.
Recall that the definition of the powersum nc-symmetric functions is given in Equation~(\ref{eq:powersum}).

\begin{thm}
If $\G = (V, \E)$ is a plurigraph, then
$$Y_{\G} = \sum_{A \subseteq \E}(-1)^{|A|}p_{\pi(A)}.$$
\label{thm:pow}
\end{thm}
\begin{proof}
We induct on $|\E|$.
If $|\E| = 0$, then $Y_{\G} = p_{1/2/\cdots/n} = p_{\pi(\emptyset)}$ and the theorem holds.
Now assume $|\E| > 0$ and let $(V,E) \in \E$ be a contraction-ready pluriedge.
As before we let $\E'$ and $\E''$ denote the multiset of pluriedges in the deletion and contraction respectively.
Then using Theorem~\ref{thm:delcongraph} we have
\begin{align*}
Y_{\G} &=  Y_{\G \setminus (V,E)} - Y_{\G / (V,E)} \uparrow^{(r_1,r_2,\dots,r_k)}\\
&= \sum_{A \subseteq \E'} (-1)^{|A|} p_{\pi(A)}  - \sum_{\bar{A} \subseteq \E''} (-1)^{|\bar{A}|} p_{\pi(\bar{A})}\uparrow^{(r_1,r_2,\dots,r_k)}\\
&= \sum_{\substack{A \subseteq \E \\ (V,E) \not\in A}} (-1)^{|A|} p_{\pi(A)}  + \sum_{\substack{A \subseteq \E\\ (V,E) \in A}} (-1)^{|A|} p_{\pi(A)}\\
&= \sum_{A \subseteq \E}(-1)^{|A|}p_{\pi(A)}.
\end{align*}
We note that
$$- \sum_{\bar{A} \subseteq \E''}(-1)^{|\bar{A}|} p_{\pi(\bar{A})}\uparrow^{(r_1,r_2,\dots,r_k)} = \sum_{\substack{B \subseteq \E\\ (V,E) \in B}} (-1)^{|B|}p_{\pi(B)}$$
since for $\bar{A} \subseteq \E''$ we have $\pi(\bar{A}) \uparrow^{(r_1,r_2,\dots,r_k)} = \pi(B)$ where $B = A \cup (V,E)$ if $A \in \E'$ corresponds to $\bar{A} \in \E''$.
Recall $\E'' = \E' / (V,E)$ and hence we have a canonical bijection between $\E'$ and $\E''$.
\end{proof}

Allowing the variables to commute we obtain a powersum expansion of the chromatic symmetric function $X_{\G}$ of plurigraph $\G$ from Theorem~\ref{thm:pow}.
When considering graphs and hypergraphs as plurigraphs as in Example~\ref{ex:graph} we obtain a powersum expansion for graphs~\cite[Theorem 2.5]{s:sfgcpg} and hypergraphs~\cite[Theorem 3.4]{s:hyper}.
Also, Lemma~\ref{lem:equiv2} will allow Theorem~\ref{thm:pow} to be applied in the setting of simplicial complexes where we get a powersum expansion for the $s$-chromatic symmetric function from~\cite[Equation (4)]{BHM}.
We now conclude this section with an example of using both Theorem~\ref{thm:delcongraph} and Theorem~\ref{thm:pow} to compute a chromatic nc-symmetric function.

\begin{ex}
Again consider the plurigraph $\G = ([4], \{([4], \{13,24\}), ([4],\{12,34\})\})$ and let $G = ([4],\{12,34\})$.
Recall that $\G \setminus G = ([4], \{([4], \{13,24\})\})$ and $\G / G = ([2], \{([2],\{12,12\})\})$.
Note $f:[4] \to \P$ is a proper coloring of $\G$  except in the following cases:
\begin{itemize}
\item $f(1) = f(2)$ and $f(3) = f(4)$
\item $f(1) = f(3)$ and $f(2) = f(4).$
\end{itemize}
It follows that
$$Y_{\G}  = \sum_{\substack{\pi \ne 1234\\ \pi \ne 12/34 \\ \pi \ne 13/24}} m_{\pi}.$$
Similarly by considering proper colorings of $\G \setminus G$ and $\G / G$ we can conclude
\begin{align*}
Y_{\G \setminus G} &= \sum_{\substack{\pi \ne 1234\\ \pi \ne 13/24}} m_{\pi} & Y_{\G / G} &= m_{1/2}  &Y_{\G /G} \uparrow^{(1,1)} = m_{12/34}.
\end{align*}
We can now directly verify Theorem~\ref{thm:delcongraph} in this case which states $Y_{\G} = Y_{\G \setminus G} - Y_{\G /G} \uparrow^{(1,1)}$.
We can alternatively use Theorem~\ref{thm:pow} to compute and obtain
$$Y_{\G} = p_{1/2/3/4} -  p_{12/34} - p_{13/24} + p_{1234}.$$
It is readily verified that the expansions of $Y_{\G}$ in the monomial basis and powersum basis describe the same element of \NCSym.
\label{ex:csf}
\end{ex}

\section{Graph-like scheduling problems}
\label{sec:sched}

Though the definition of scheduling problems only includes atomic formulas with weak inequalities we can build strict inequality, equality, and nonequality as follows
\begin{align*} 
(x_i < x_j) &= \neg(x_j \leq x_i) & (x_i = x_j) &= (x_i \leq x_j) \wedge (x_j \leq x_i) & (x_i \neq x_j) &= \neg(x_i = x_j).
\end{align*}
A boolean formula $C$ is called \emph{edge-like} if it can be expressed as a disjunction of nonequalities.
That is, $C$ is edge-like if
$$C = \bigvee_{(i,j) \in I} (x_i \ne x_j)$$
for some $I \subseteq [n] \times [n]$.
A boolean formula $S$ is called \emph{graph-like} if it can be expressed as a conjunction of edge-like boolean formulas.
That is, $S$ is graph-like if
$$S = \bigwedge_{\alpha \in I} C_\alpha$$
for some finite index set $I$ where $C_{\alpha}$ is edge-like for each $\alpha \in I$.

\begin{ex}
For a plurigraph $\G = (V, \E)$ define the scheduling problem $S_{\G}$ by
$$S_{\G} := \bigwedge_{(V,E) \in \E} \bigvee_{uv \in E} (x_u \ne x_v).$$
Here $S_{\G}$ is the scheduling problem of properly coloring $\G$ where $x_v$ represents the color given to $v \in V$.
The scheduling problem $S_{\G}$ is a graph-like scheduling problem.
\label{ex:gsched}
\end{ex}

Our next result shows that any graph-like scheduling problem is equivalent to properly coloring some plurigraph, and hence for any graph-like scheduling problem the scheduling nc-quasisymmetric function is the chromatic nc-symmetric function for some plurigraph.
In particular $\S_S$ lies in $\NCSym$ whenever $S$ is graph-like.
The fact $\S_S$ is nc-symmetric when $S$ is graph-like is not surprising.
It follows immediately from the fact that $(x_i \ne x_j)$ is symmetric in $i$ and $j$.

\begin{thm}
A scheduling problem $S$ is graph-like if and only if $S = S_{\G}$ for some plurigraph $\G$. 
\label{thm:gen}
\end{thm}
\begin{proof}
Let $S$ be a graph-like scheduling problem on $n$ elements where $S = C_1 \wedge C_2 \wedge \cdots \wedge C_m$ with each $C_i$ edge-like, where
$$C_i = \bigvee_{(j,k) \in I_i} (x_j \ne x_k)$$
for some $I_i \subseteq [n] \times [n]$.
We define the plurigraph $\G = ([n], \E)$ where 
$$\E = \{([n], E_1), ([n], E_2), \dots, ([n], E_m)\}$$
 and 
$$E_i = \{jk : (j,k) \in I_i\}.$$
We then have $S = S_{\G}$.
The reverse direction is shown in Example~\ref{ex:gsched}.
\end{proof}

Theorem~\ref{thm:gen} has the following immediate corollary.

\begin{cor}
If $S$ is a graph-like scheduling problem, then $\S_S = Y_{\G}$ for some plurigraph $\G$.
\label{cor:gen}
\end{cor}

Let $S$ be any scheduling problem with some fixed expression and let $V(S)$ be the set of $i$ such that $x_i$ appears in $S$.
If $C$ is edge-like, then $C = \bigvee_{(i,j) \in I} (x_i \ne x_j)$ for some $I \subseteq [n] \times [n]$.
Thus $\neg C = \bigwedge_{(i,j) \in I} (x_i = x_j)$, and we say that $i \sim_C j$ if $(x_i = x_j)$ appears in $\neg C$.
We get an equivalence relation on $V(C)$ by taking the reflexive transitive closure of $\sim_C$.
We will write $V(C) = O_1 \uplus O_2 \uplus \cdots \uplus O_k$ to denote the decomposition of $V(C)$ into $\sim_C$ equivalence classes.

We define a contraction operation on scheduling problems.
Let $S$ be a boolean formula over the atomic formulas $(x_i \leq x_j)$ for $i,j \in [n]$.
Given positive integers $r$ and $t$ with $r+t < n$ we define $S\downarrow_{r,t}$ to be the boolean formula over atomic formulas $(x_i \leq x_j)$ for $i,j \in [n-r]$ obtained from $S$ where $x_{n-t}, x_{n-t-1}, \dots, x_{n-t-r}$ are all identified to $x_{n-r-t}$ and then the variable indices are standardized to lie in $[n-r]$.
For $(r_1, r_2, \dots, r_k)$ a sequence of positive integers with $r = \sum_{i=1}^k r_i$ and $r+k < n$.
We define $S\downarrow_{(r_1, r_2, \dots, r_k)} := S\downarrow_{(r_1, r_2, \dots, r_k), 0}$ where we have the recursion $S\downarrow_{(r_1, r_2, \dots, r_j), t} = \left(S\downarrow_{r_1,t}\right)\downarrow_{(r_2, r_3, \dots, r_j), t+1}$ with $S\downarrow_{(),t} = S$.
We now give an example of contraction of a scheduling problem and set composition.

\begin{ex}
\begin{align*}
&\left(((x_1 \leq x_2) \wedge (x_1 < x_3)) \vee ((x_3 \neq x_4) \wedge (x_4 \leq x_5))\right)\downarrow_{(2,1)}\\
= &\left(((x_1 \leq x_2) \wedge (x_1 < x_3)) \vee ((x_3 \neq x_4) \wedge (x_4 \leq x_5))\right)\downarrow_{(2,1),0}\\
= &\left(((x_1 \leq x_2) \wedge (x_1 < x_3)) \vee ((x_3 \neq x_3) \wedge (x_3 \leq x_3))\right)\downarrow_{(1),1}\\
= &((x_1 \leq x_1) \wedge (x_1 < x_2)) \vee ((x_2 \neq x_2) \wedge (x_2 \leq x_2))\\
= &(x_1 < x_2)
\end{align*}
\end{ex}

\begin{lem}
If $S = S' \wedge C$ is scheduling problem, then $\S_S = \S_{S'} - \S_{S' \wedge \neg C}$.
\label{lem:neg}
\end{lem}
\begin{proof}
First note that
$$\{\Phi : \Phi \mathrm{\;solves\;} S'\} = \{\Phi : \Phi \mathrm{\;solves\;} S' \wedge C\} \uplus \{\Phi : \Phi \mathrm{\;solves\;} S' \wedge \neg C\}.$$
The scheduling nc-quasisymmetric function is just a sum of monomial nc-quasisymmetric functions for set compositions solving the scheduling problem.
It follows that
$$\S_{S'} = \S_{S' \wedge C} + \S_{S' \wedge \neg C} = \S_{S} + \S_{S' \wedge \neg C},$$
and by arranging this equation we obtain $\S_S = \S_{S'} - \S_{S' \wedge \neg C}$ as desired.
\end{proof}

Recall for two disjoint subsets $A, B \subseteq [n]$ we say $A < B$ if $a < b$ for all $a \in A$ and $b \in B$.
Given a scheduling problem $S$ on $n$ elements with $S = S' \wedge C$ we call $C$ a \emph{contractible clause} if:
\begin{itemize}
\item $C$ is edge-like,
\item $V(C) = \{x_{n-s}, x_{n-s+1}, \dots, x_n\}$ for some $0 < s \leq n$,
\item There exists an ordering so that $O_i > O_j$ for $i < j$ where 
$$V(C) = O_1 \uplus O_2 \uplus \cdots \uplus O_k$$
and the $O_i$ are the $\sim_C$ equivalence classes.
\end{itemize}
Note that any edge-like clause $C$ can be made into a contractible clause by relabeling the variables $x_i$ if needed.
The relabeling works similarly to the relabeling for plurigraphs in the previous section.

The contraction operation for scheduling problems define earlier operates on the variables $x_1, x_2, \dots, x_n$.
We could have also defined contraction to operate on the set of indices $[n]$.
Viewing contraction as operating on $[n]$, contraction can be applied to set compositions of $[n]$ provided that all elements being identified are in the same block.
Before proving the next lemma we give an example of when contraction can and cannot be applied to set compositions.
We give this example instead of a formal definition of contraction for set compositions because (when it is defined) contraction for set compositions has the same rule as contraction for scheduling problems.

\begin{ex}
Let us consider set compositions of $[5]$ and take the sequence of positive integers $(2,1)$.
For the set composition $(125, 34)$, contraction with respect to the sequence $(2,1)$ is not defined.
To perform this contraction we would first identify the elements $3$, $4$, and $5$.
This would not lead to a valid set composition of $[3]$ as the elements $3$, $4$, and $5$ occupy more than one block.
However, we can perform contraction with respect to the sequence $(2,1)$ on the set composition $(12,345)$ to obtain
\begin{align*}
(12,345)\downarrow_{(2,1)} &= (12,345)\downarrow_{(2,1),0}\\
&= (12,3)\downarrow_{(1),1}\\
&= (1,2).
\end{align*}
\end{ex}

\begin{lem}
If $S = S' \wedge C$ is a scheduling problem and $C$ is a contractible clause, then
$$\S_{S' \wedge \neg C} = \left(\S_{S'\downarrow_{(r_1,r_2,\dots,r_k)}}\right)\uparrow^{(r_1,r_2,\dots,r_k)}$$
where $V(C) = O_1 \uplus O_2 \uplus \cdots \uplus O_k$ is the decomposition into $\sim_C$ equivalence classes and $r_i = |O_i| - 1$.
\label{lem:edge}
\end{lem}
\begin{proof}
We must show that
$$\{\Phi : \Phi \mathrm{\;solves\;} S' \wedge \neg C\} = \{\Phi\uparrow^{(r_1,r_2,\dots,r_k)} : \Phi \mathrm{\;solves\;} S'\downarrow_{(r_1,r_2,\dots,r_k)}\}.$$
Note that $\Phi$ solving $\neg C$ exactly means that $i$ and $j$ are  in the same block whenever $x_i \sim_C x_j$.
So, $\Phi$ solving $S' \wedge \neg C$ is equivalent to $\Phi$ solving $S'$ with $i$ and $j$ in the same block of $\Phi$ whenever $x_i \sim_C x_j$.

For any $\Phi$ solving $S'\downarrow_{(r_1,r_2,\dots,r_k)}$ we will have $i$ and $j$ in the same block of $\Phi\uparrow^{(r_1,r_2,\dots,r_k)}$ whenever $x_i \sim_C x_j$.
This follows immediately from the definition of induction and the fact that $C$ is a contractible clause.
It remains the verify that $\Phi\uparrow^{(r_1,r_2,\dots,r_k)}$ solves $S'$ for any $\Phi$ solving $S'\downarrow_{(r_1,r_2,\dots,r_k)}$.
This also follows immediately from the definitions of induction and contraction.

Conversely, given $\Phi$ solving $S' \wedge \neg C$ we must verify that $\Phi = \Psi\uparrow^{(r_1,r_2,\dots,r_k)}$ for some $\Psi$ solving $S'\downarrow_{(r_1,r_2,\dots,r_k)}$.
If $\Phi$ solves $S' \wedge \neg C$, then in particular $\Phi$ solves $\neg C$.
Since $C$ is a contractible clause, contraction with respect to the sequence $(r_1,r_2,\dots,r_k)$ is defined for any $\Phi$ solving $\neg C$.
It follows that if $\Phi$ solves $S' \wedge \neg C$, then $\Phi = (\Phi\downarrow_{(r_1,r_2,\dots,r_k)})\uparrow^{(r_1,r_2,\dots,r_k)}$ and $\Phi\downarrow_{(r_1,r_2,\dots,r_k)}$ solves $S'\downarrow_{(r_1,r_2,\dots,r_k)}$.
\end{proof}

We now give a deletion-contraction law that applies to any scheduling problem $S$ that can be expressed as $S = S' \wedge C$ where $C$ is edge-like.
The deletion-contraction law follows directly from Lemma~\ref{lem:neg} and Lemma~\ref{lem:edge}.

\begin{thm}
If $S = S' \wedge C$ if a scheduling problem and $C$ is a contractible clause, then $$\S_S = \S_{S'} -  \left(\S_{S'\downarrow_{(r_1,r_2,\dots,r_k)}}\right)\uparrow^{(r_1,r_2,\dots,r_k)}$$
where $V(C) = O_1 \uplus O_2 \uplus \cdots \uplus O_k$ is the decomposition into $\sim_C$ equivalence classes and $r_i = |O_i| - 1$.
\label{thm:delcon}
\end{thm}

If $S = C_1 \wedge C_2 \wedge \cdots \wedge C_m$ is graph-like with each $C_i$ edge-like, then we can apply Theorem~\ref{thm:delcon} with $S' = C_1 \wedge C_2 \wedge \cdots \wedge C_{m-1}$ and $C = C_m$.
In this case both $S'$ and $S'\downarrow_{(r_1,r_2,\dots,r_k)}$ are graph-like and so Theorem~\ref{thm:delcon} can be iterated (after perhaps relabeling).
If $S$ is a graph-like scheduling problem, then by Theorem~\ref{thm:gen} we have that $S = S_{\G}$ for a plurigraph $\G$.
In that case the deletion-contraction for $\S_S$ in Theorem~\ref{thm:delcon} is the same as the deletion-contraction for $Y_{\G}$ in Theorem~\ref{thm:delcongraph}.
We now given an example of using Theorem~\ref{thm:delcon} to compute a scheduling nc-quasisymmetric function.

\begin{ex}
We let $S = S' \wedge C$ where 
\begin{align*}
S' &= (x_1 \leq x_2) \wedge (x_2 \leq x_3) \wedge (x_3 \leq x_4) &C &= (x_1 \ne x_2) \vee (x_3 \ne x_4).
\end{align*}
We then let 
$$S'' = S' \downarrow_{(1,1)} = (x_1 \leq x_1) \wedge (x_1 \leq x_2) \wedge (x_2 \leq x_2) = (x_1 \leq x_2).$$
Computing the scheduling nc-quasisymmetric functions we see
\begin{align*}
\S_{S'} &= M_{(1234)} + M_{(1,234)} + M_{(12,34)} + M_{(123,4)} +M_{(1,2,34)}\\
&\quad\quad\quad + M_{(1,23,4)} + M_{(12,3,4)} + M_{(1,2,3,4)}\\
\S_{S''} &= M_{(12)} + M_{(1,2)}\\
\S_{S''} \uparrow^{(1,1)} &= M_{(1234)} + M_{(12,34)}.
\end{align*}
We then can apply Theorem~\ref{thm:delcon}
\begin{align*}
\S_S &= \S_{S'} - \S_{S''} \uparrow^{(1,1)} \\
&=M_{(1,234)} + M_{(123,4)} +M_{(1,2,34)} + M_{(1,23,4)} + M_{(12,3,4)} + M_{(1,2,3,4)}.
\end{align*}
\end{ex}

We conclude this section with discussion of a geometric interpretation of the results in this section.
Consider $S = S' \wedge C$ a scheduling problem on $n$ elements where $C$ is a contractible edge-like clause.
Let $C$ be given by
$$C = \bigvee_{(i,j) \in I} (x_i \ne x_j)$$
for some $I \subseteq [n] \times [n]$.
Also let $V(C) = O_1 \uplus O_2 \uplus \cdots \uplus O_k$ be the decomposition into $\sim_C$ equivalence classes.
As before we set $r_i = |O_i| - 1$ and $r = \sum_{i=1}^k r_i$.
Lastly let 
$$V_i = \{(x_1, x_2, \dots, x_n) \in \R^n : x_j = x_{j'} \mathrm{\;for\;all\;} (j,j') \in O_i \times O_i\}.$$
So, $V_i$ is a subspace of $\R^n$ of codimension $r_i$ defined by the equivalence class $O_i$.

Each solution to the scheduling problem $S$ can be thought of as a positive lattice point in $\P^n \subset \R^n$.
By Lemma~\ref{lem:neg} we know that to determine the solutions to $S = S' \wedge C$ it suffices to find the solutions of $S'$ and $S' \wedge \neg C$.
Now $S'$ is a less restricted scheduling problem on $n$ elements, and so solutions to $S'$ again can be thought of as lattice points in $\P^n \subset \R^n$.
While $S' \wedge \neg C$ is a scheduling problem on $n$ elements, by Lemma~\ref{lem:edge} we can actually consider the scheduling problem $S'\downarrow_{(r_1,r_2,\dots,r_k)}$ which is a scheduling problem on $n - r$ elements.
So, solutions to $S'\downarrow_{(r_1,r_2,\dots,r_k)}$ can be represented by lattice points in $\P^{n-r} \subset \R^{n-r}$ where $\R^{n-r} \cong \bigcap_{i=1}^k V_i$.
Geometrically we see the deletion-contraction law in Theorem~\ref{thm:delcon} takes $S = S' \wedge C$ and reduces it to the easier problems $S'$ and $S'\downarrow_{(r_1,r_2,\dots,r_k)}$.
The problem $S'$ will be easier in the sense that there are less constraints, and $S'\downarrow_{(r_1,r_2,\dots,r_k)}$ is a lower dimensional problem.
Moreover, $\neg C$ is saying that the solutions occur  on the intersection of some hyperplanes, and $S'\downarrow_{(r_1,r_2,\dots,r_k)}$ is a lower dimensional problem occurring on the intersection of those hyperplanes.

\section{Distinguishing plurigraphs and hypertrees}
\label{sec:dist}
In this section we show that when restricted to a certain class of plurigraphs, the chromatic nc-symmetric function distinguishes these plurigraphs up to isomorphism.
Later we will look at the power of the chromatic symmetric function in distinguishing hypergraphs, and we will pay particular attention to the case of hypertrees where we provide an example of uniform hypertrees which are not isomorphic yet have the same chromatic symmetric function.

\subsection{Plurigraphs}

Two plurigraphs $\G_1 = (V_1, \E_1)$ and $\G_2 = (V_2, \E_2)$ are \emph{isomorphic} if there exists a bijection $\phi: V_1 \to V_2$ and a bijection $\psi: \E_1 \to \E_2$ such that $\phi$ is an isomorphism of the graphs $G$ and $\psi(G)$ for all $G \in \E_1$.
When using the constructions in Example~\ref{ex:graph} this definition of isomorphic agrees with the usual definition of isomorphic for graphs and hypergraphs.
For a fixed vertex set $V,$ let $\Graph (V)$ denote the set of all graphs with vertex set $V$.
We again allow loops and multiple edges.
We call a graph $G$ a \emph{simple} graph if $G$ contains no loops or multiple edges. 
Define a preorder $(\Graph(V), \prec)$ as follows, for any $G_1, G_2 \in \Graph(V)$ we have $G_1 \prec G_2$ if and only if every connected component of $G_1$ is contained in some connected component of $G_2$.
We call a plurigraph $\G = (V, \E)$ a \emph{simple plurigraph} if:
\begin{enumerate}
\item[(i)] For each $G \in \E$, every connected component of $G$ is a complete graph.
\item[(ii)] There does not exist $G_1, G_2 \in \E$ with $G_1 \neq G_2$ and $G_1 \prec G_2$.
\end{enumerate}

Recall that for a plurigraph $\G$ the chromatic nc-symmetric function is the scheduling nc-symmetric function for the corresponding scheduling problem $S_{\G}$ from Example~\ref{ex:gsched}.
So, for two plurigraphs $\G_1$ and $\G_2$ we have $Y_{\G_1} = Y_{\G_2}$ if and only if $\S_{S_{\G_1}} = \S_{S_{\G_2}}$ which occurs if and only if the boolean formulas $S_{\G_1}$ and $S_{\G_2}$ have exactly the same collections of solutions.
If $\G = (V, \E)$ is not a simple plurigraph because (ii) is violated and $G_1 \prec G_2$ for $G_1, G_2 \in \E$ where $G_1 \neq G_2$, then $Y_{\G} = Y_{\G \setminus G_2}$ since $S_{\G} = S_{\G \setminus G_2}$ in this case.
Notice that $f: V \to \P$ is a solution to the scheduling problem $S_{\G}$ if and only if  for each $G \in \E$ the map $f$ is not monochromatic on some component of $G$.
Thus if each component of $G_1$ is contained in some component of $G_2$ the conditions imposed by $G_2$ are redundant, and its removal will not change the scheduling problem.
Also observe that adding or removing edges in any $G \in \E$ which do not combined or break connected components will not change the scheduling problem $S_{\G}$.
We now show that the chromatic nc-symmetric function is a complete invariant if we restrict the simple plurigraphs.

\begin{prop}
If $\G$ is a simple plurigraph, then given $Y_{\G}$ we can construct $\G$.
\label{prop:YG}
\end{prop}
\begin{proof}
Let $\G = (V, \E)$ with $V = [n]$, then by Theorem~\ref{thm:pow}
$$Y_{\G} = \sum_{A \subseteq \E}(-1)^{|A|}p_{\pi(A)}.$$
We must determine the pluriedges in $\E$.
Under the assumption that $\G$ is a simple plurigraph we see that the elements of $\E$ will exactly correspond the (nontrivial) minimal elements of $\{\pi \in \Pi_n: [p_{\pi}]Y_{\G} \neq 0\}$ where $[p_{\pi}]Y_{\G}$ denotes the coefficient of $p_{\pi}$ in $Y_{\G}$.
For each $G \in \E$ we have $A = \{G\} \subseteq \E$ and $p_{\pi(A)}$ occurs with coefficient $-1$ as this is the only way to obtain this partition when $\G$ is a simple plurigraph.
For $A \subseteq \E$ with $|A| > 1$, then $\pi(A)$ is strictly above $\pi(\{G\})$ in $\Pi_n$ for any $G \in A$ since $\G$ is a simple plurigraph.
\end{proof}

This Proposition means that the chromatic nc-symmetric function distinguishes simple graphs which was originally shown in~\cite[Proposition 8.2]{GS}.
Also, this Proposition means that the chromatic nc-symmetric function distinguishes among hypergraphs whose hyperedge sets are antichains in the boolean algebra.
An alternative phrasing of Proposition~\ref{prop:YG} would be to say that given a plurigraph $\G = (V, \E)$ we can determine the connected components of the minimal pluriedges from $Y_{\G}$.
This result can then be interpreted in the lattice of partitions.

Again let $\Pi_n$ be the lattice of partitions of $[n]$ partially ordered by refined.
This lattice has bottom element $\hat{0} = 1/2/\cdots/n.$
For $\pi_1, \pi_2 \in \Pi_n$ we denote the \emph{join} or \emph{least upper bound} of $\pi_1$ and $\pi_2$ by $\pi_1 \vee \pi_2 \in \Pi_n$.
Here there is an unfortunate case of double notation. 
However, $\vee$ is the standard accepted notation for both logical disjunction of boolean formulas as well as for the join of lattice elements.
We will use $\vee$ to denote the join for the remainder of this section.
As an example for $123/4/5, 1/234/5 \in \Pi_5$ we have
$$123/4/5 \vee 1/234/5 = 1234/5.$$

For any graph $G = ([n], E)$ we let $\pi(G) \in \Pi_n$ denote the partition of $[n]$ given by the connected components of $G$.
Then any plurigraph $\G = ([n], \E)$ determines the collection $\{\pi(G) : G \in \E\}$.
In this way the simple plurigraphs defined above are in bijective correspond with antichains in $\Pi_n$ (with the exception of the antichain $\{\hat{0}\}$).
Given any $\A \subseteq \Pi_n$ we can consider the nc-symmetric function
$$Y_{\A} = \sum_{A \subseteq \A} (-1)^{|A|} p_{\pi(A)}$$
where $\pi(A) = \bigvee_{\pi \in A} \pi$.
The nc-symmetric function $Y_{\A}$ is then the chromatic nc-symmetric function $Y_{\G}$ when $\A = \{\pi(G) : \G \in \E\}$.
As an example we can take $\A = \{13/24, 12/34\}$, and then
$$Y_{\A} = p_{1/2/3/4} -  p_{12/34} - p_{13/24} + p_{1234}$$
since $13/24 \vee 12/34 = 1234$.
Notice that $\A = \{\pi(G) : \G \in \E\}$ if $\G = (V, \E)$ is the plurigraph shown in Figure~\ref{fig:example}, and $Y_{\A} = Y_{\G}$ where $Y_{\G}$ was computed in Example~\ref{ex:csf}.

\subsection{Hypergraphs and hypertrees}

We have seen that the chromatic nc-symmetric function is a powerful invariant, and will now investigate the chromatic symmetric function.
It is an open problem, first considered in~\cite{s:sfgcpg}, to determine if the chromatic symmetric function distinguishes trees up to isomorphism.
For some partial results on this problem see~\cite{MMW,PZ}.
Russel has verified that the chromatic symmetric function distinguishes trees on $25$ or fewer vertices up to isomorphism~\cite{github}.
We will investigate the analogous question for uniform hypertrees.
We note that the chromatic symmetric function of hypertrees has also been studied by Taylor in~\cite{Tay17} where an expansion in Gessel's fundamental basis of quasisymmetric functions is given for hypertrees with prime sized hyperedges.

Throughout this section let $H = (V,E)$ be a hypergraph on $n$ vertices.
Let $E$ be a set of subsets of $V$ and assume that $|e| > 1$ for all $e \in E$.
That is, we do not allow multiple hyperedges or hyperedges of size $1$ in our hypergraphs.
A \emph{walk} of length $\ell > 0$  between $v_1 \in V$ and $v_{\ell} \in V$ is a  sequence
$$(v_1, e_1, v_2, e_2, \dots, v_{\ell}, e_{\ell}, v_{\ell + 1})$$
such that $e_i \in E$ with $v_i, v_{i+1} \in e_i$ for all $i$.
If all the vertices and hyperedges are distinct, then the walk is called a \emph{path}.
In the case all vertices and hyperedges are distinct with the exception that $v_1 = v_{\ell + 1}$ we call the walk a \emph{cycle}.
The hypergraph $H$ is \emph{connected} if for any $v, v' \in V$ there exists a path between $v$ and $v'$.
A \emph{hypertree} is a connected hypergraph with no cycles.
We call $H$ a \emph{linear hypergraph} if $|e_1 \cap e_2| \leq 1$ for all $e_1, e_1 \in E$ such that $e_1 \neq e_2$.
Notice a hypertree is necessarily linear, otherwise for distinct hyperedges $e_1, e_2 \in E$ and distinct vertices $v_1, v_2 \in e_1 \cap e_2$ there is a cycle a length 2
$$(v_1, e_1, v_2, e_2, v_1).$$

For $H$ let $(a_i)_{i=2}^n$ be the sequence defined by
$$a_i := |\{e \in E : |e| = i\}|$$
which records the number of hyperedges of each size in the hypergraph.
The \emph{hyperedge magnitude} of $H$ is defined to be the sum
$$\sum_{i=2}^n (i-1)a_i.$$
In~\cite{GK05} is it shown that a connected hypergraph on $n$ vertices is a hypertree if and only if the hyperedge magnitude is $n-1$.
We give the following lemma which extends this result and is a generalization of the corresponding well known fact for trees.

\begin{lem}
Let $H = (V, E)$ be a hypergraph on $n$ vertices, and consider the following conditions:
\begin{enumerate}
\item[(i)] $H$ is connected.
\item[(ii)] $H$ is acyclic.
\item[(iii)] $H$ has hyperedge magnitude equal $n-1$.
\end{enumerate}
Any two of the above conditions together imply the third.
Hence, to show that a hypergraph $H$ is a hypertree is suffices to prove that any two of the above conditions hold for $H$.
\end{lem}
\begin{proof}
From~\cite{GK05} we already know that (i) and (ii) together imply (iii), and also that (i) and (iii) together imply (ii).
It remains to show that (ii) and (iii) together imply (i).
Assume that $H = (V,E)$ is a acylic hypergraph on $n$ vertices with hyperedge magnitude equal $n-1$.
We order the hyperedges $E = \{e_1, e_2, \dots, e_m\}$ and let $H_i = (V, E_i)$ for where $E_i = \{e_1, e_2, \dots, e_i\}$ for $1 \leq i \leq m$.
Also let $H_0 = (V, \emptyset)$
Notice $H_i$ will be an acyclic hypergraph for $1 \leq i \leq m$.
Since each hypergraph is acyclic it follows that if $H_i$ has $c$ connected components, then $H_{i+1}$ has $c - |e_{i+1}| + 1$ connected components.
Now $H_0$ has $n$ connected components and so it follows that $H = H_m$ has $c$ connected components where
$$c = n - \sum_{i=1}^m (|e_i| - 1) = n - (n-1) = 1.$$
Here we have used the assumption that $H$ has edge magintude $n-1$.
Therefore we have shown $H$ is connected and completed the proof.
\end{proof}

For a graph $G$ or hypergraph $H$ we let $P_G(t)$ and $P_H(t)$ denote the \emph{chromatic polynomial} of $G$ and $H$ respectively.
If $G$ is a graph on $n$ vertices, then $G$ is a tree if and only if $P_G(t) = t(t-1)^{n-1}$.
There is a similar result for $s$-uniform hypertrees when we restrict to linear hypergraphs.
It is proven in~\cite[Theorem 5]{BL07} that if $H$ is a linear hypergraph on $n$ vertices, then $H$ is an $s$-uniform hypertree with $m$ hyperedges if and only if $P_H(t) = t(t^{s-1}-1)^m$.
Here we observe some similar behavior between trees and uniform hypertrees when we restrict to linear hypergraphs.
In what follows we will show some of the results on the chromatic symmetric which can be proven from trees can also be proven for uniform hypertrees.
However, we will also exhibit two $3$-uniform hypertrees which are not isomorphic yet have the same chromatic symmetric function.

We now give a formula for the chromatic  symmetric function of a hypergraph.
This formula is very close to~\cite[Theorem 3.4]{s:hyper} and can be easily obtain by letting variables commute in the powersum expansion of the chromatic nc-symmetric function of a hypergraph.
The algebra of symmetric functions has the powersum basis which is indexed by \emph{integer partitions}.
Given any set partition $\pi = B_1/B_2/ \cdots /B_{\ell}$ of $[n]$ we let $\type \pi$ be the integer partition of $n$ given by the sizes of the blocks in $\pi$.
Given $A \subseteq E$ we let $\lambda(A) = \type \pi(A)$.
The chromatic symmetric function $X_H$ has the expansion
$$X_H = \sum_{A \subseteq E} (-1)^{|A|} p_{\lambda(A)}.$$
Let $c_{\lambda(H)}$ denote the coefficient of $p_{\lambda}$ is the powersum expansion of $X_H$ so that
$$X_H = \sum_{\lambda} c_{\lambda(H)} p_{\lambda},$$
and let $c_{i}(H) = c_{(i,1,1,\dots,1)}(H)$.
Notice that $X_H$ is homogeneous of degree $|V|$ and when $H$ is $s$-uniform $-c_s(H) = |E|$.
Thus, we can always recover the number of vertices from $X_H$, and we can recover the number of hyperedges in the case of uniform hypergraphs.

Now assume that $H$ is $s$-uniform and acyclic.
For every $A \subseteq E$ the hypergraph $(V,A)$ has $n - (s-1)|A|$ connected components.
For any integer partition $\lambda$ we let $\len \lambda$ denote the length of the partition.
Thus for $s$-uniform acyclic hypergraphs 
$$\len \lambda(A) = n - (s-1)|A|$$
for any $A \subseteq E$.
It then follows
$$c_{\lambda} = (-1)^{\frac{n-k}{s-1}} |\{A \subseteq E : \lambda(A) = \lambda\}|$$
for $\lambda \vdash n$ with $\len \lambda = k$.
This implies the relation
\begin{equation*}
(-1)^{\frac{n-k}{s-1}}\sum_{\substack{\lambda \vdash n\\ \len \lambda = k}} c_{\lambda}(H) = \binom{m}{\frac{n-k}{s-1}}
\end{equation*}
where $m = |E|$.

For a vertex $v \in V$, the \emph{degree} of $v$ in $H$ is $\deg v := |\{e \in E : v \in e\}|$.
The \emph{degree sequence} of $H$ is the collection of the degrees of all vertices of $H$ arranged in weakly decreasing order.
Our next result shows that the chromatic symmetric function of a uniform hypertree determines its degree sequence.
In~\cite[Corollary 5]{MMW} it was shown that the chromatic symmetric function determines the degree sequence of a tree.

\begin{prop}
If $H$ is a uniform hypertree, then the degree sequence of $H$ can be determined from $X_H$.
\end{prop}
\begin{proof}
Let $H = (V,E)$ be an $s$-uniform hypertree on $n$ vertices.
Thus $H$ must have $m = \tfrac{n-1}{s-1}$ hyperedges.
Let $X_H = \sum c_{\lambda} p_{\lambda}$ and let $D_i$ denote the number of vertices of a degree $i$ in $H$.
It suffices to show that we can determine the numbers $D_i$ for $1 \leq i \leq m$.
Since $H$ is a hypertree and hence connected, we must have $D_0 = 0$.
For any $\lambda \vdash n$ let $1(\lambda)$ denote the number of parts of size $1$ in $\lambda$.
Recall that if $A \subseteq E$, then $\len \lambda(A) = n - (s-1)|A|$.
Any $1$ in the partition $\lambda(A)$ must come from a vertex of degree at most $m - |A|$.
Now for any integer $0 \leq i \leq m$ let us consider partitions $\lambda$ with $\len \lambda = k_i$ where $k_i = n - (s-1)(m-i)$.
Exactly the vertices of $H$ of degree at most $i$ will contribute to the sum
$$(-1)^{m-i}\sum_{\substack{\lambda \vdash n \\ \len \lambda = k_i}} c_{\lambda} \cdot 1(\lambda).$$
Note that a vertex of degree $j$ will contribute to the sum exactly $\binom{m-j}{i-j}$ times.
It follows that
$$(-1)^{m-i}\sum_{\substack{\lambda \vdash n \\ \len \lambda = k_i}} c_{\lambda} \cdot 1(\lambda) = \sum_{j=1}^i \binom{m-j}{i-j} D_j.$$
This gives a triangular system that we can solve for each $D_i$.
Therefore $X_H$ determines the degree sequence of a hypertree $H$.
\end{proof}

We conclude this section by showing that the chromatic symmetric function is not a complete invariant among uniform hypertrees.
We give two pairs of $3$-uniform hypertrees on $21$ vertices which are not isomorphic, but have the same chromatic symmetric function.
These hypertrees were found by using nauty~\cite{nauty} to enumerate all $3$-uniform hypertrees up to isomorphism and then using SageMath~\cite{sagemath} to compute the chromatic symmetric functions.
For completeness the SageMath code which can be used to test this is included the Appendix.
The computation indicates that the examples are minimal.
That is, there does not exist a pair of hypertrees on fewer than $21$ vertices which are not isomorphic but have the same chromatic symmetric function.
Let $H_1 = (V, E_1)$, $H_2 = (V, E_2)$, $H_3 = (V, E_3)$, and $H_4 = (V, E_4)$ where $V = \{0,1,\dots, 20\}$ and
\begin{align*}
E_1 &= \{\{0, 1, 2\}, \{0, 3, 4\}, \{1, 5, 6\}, \{0, 7, 8\}, \{2, 9, 10\}, \{1, 11, 12\}, \{9, 13, 14\},\\
 & \qquad\qquad \{16, 3, 15\}, \{17, 18, 7\}, \{19, 20, 13\}\} \\
\\
E_2 &= \{\{0, 1, 2\}, \{0, 3, 4\}, \{1, 5, 6\}, \{0, 7, 8\}, \{2, 9, 10\}, \{1, 11, 12\}, \{9, 13, 14\},\\
& \qquad\qquad \{16, 3, 15\}, \{17, 18, 5\}, \{19, 20, 15\}\} \\
\\
E_3 &= \{\{0, 1, 2\}, \{0, 3, 4\}, \{1, 5, 6\}, \{0, 7, 8\}, \{5, 9, 10\}, \{5, 11, 12\}, \{0, 13, 14\}, \\
& \qquad\qquad \{16, 2, 15\}, \{1, 17, 18\}, \{19, 20, 15\}\} \\
\\
E_4 &=  \{\{0, 1, 2\}, \{0, 3, 4\}, \{1, 5, 6\}, \{0, 7, 8\}, \{2, 9, 10\}, \{1, 11, 12\}, \{0, 13, 14\}, \\
& \qquad\qquad \{16, 9, 15\}, \{17, 18, 9\}, \{3, 19, 20\}\}.
\end{align*}
One can check that $H_1$, $H_2$, $H_3$, and $H_4$ are all $3$-uniform hypertrees on $21$ vertices and that $X_{H_1} = X_{H_2}$ and $X_{H_3} = X_{H_4}$
However, $H_1$ is not isomorphic to $H_2$ and $H_3$ is not isomorphic to $H_4$.
The hypertrees $H_1$ and $H_2$ are shown in Figure~\ref{fig:H1H2}.
The hypertrees $H_3$ and $H_4$ are shown in Figure~\ref{fig:H3H4}.

\begin{figure}
\centering
\includegraphics[scale = 0.8]{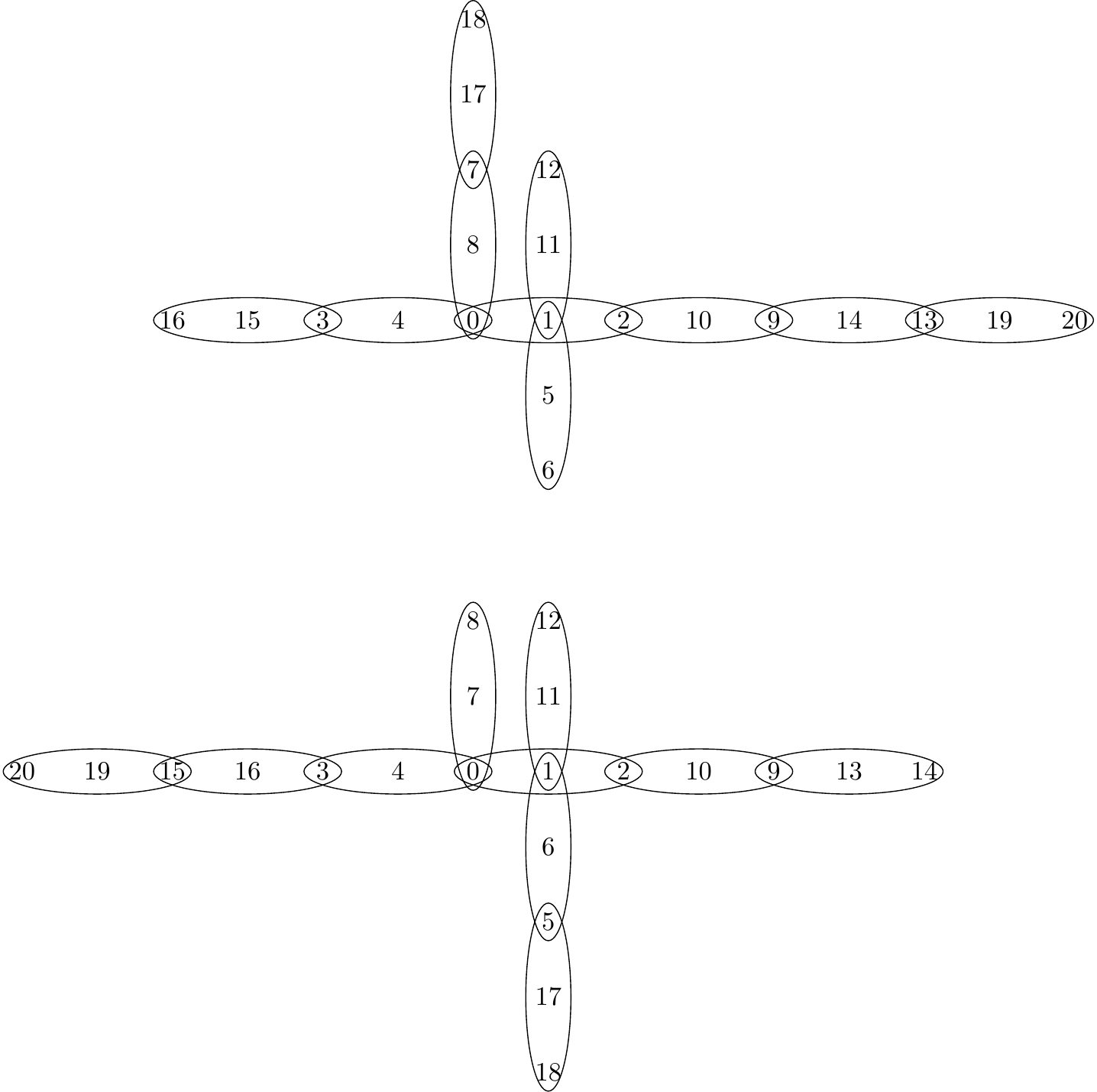}
\caption{The hypertree $H_1$ above and the hypertree $H_2$ below which are not isomorhpic but have the same chromatic symmetric function.}
\label{fig:H1H2}
\end{figure}

\begin{figure}
\centering
\includegraphics[scale = 0.8]{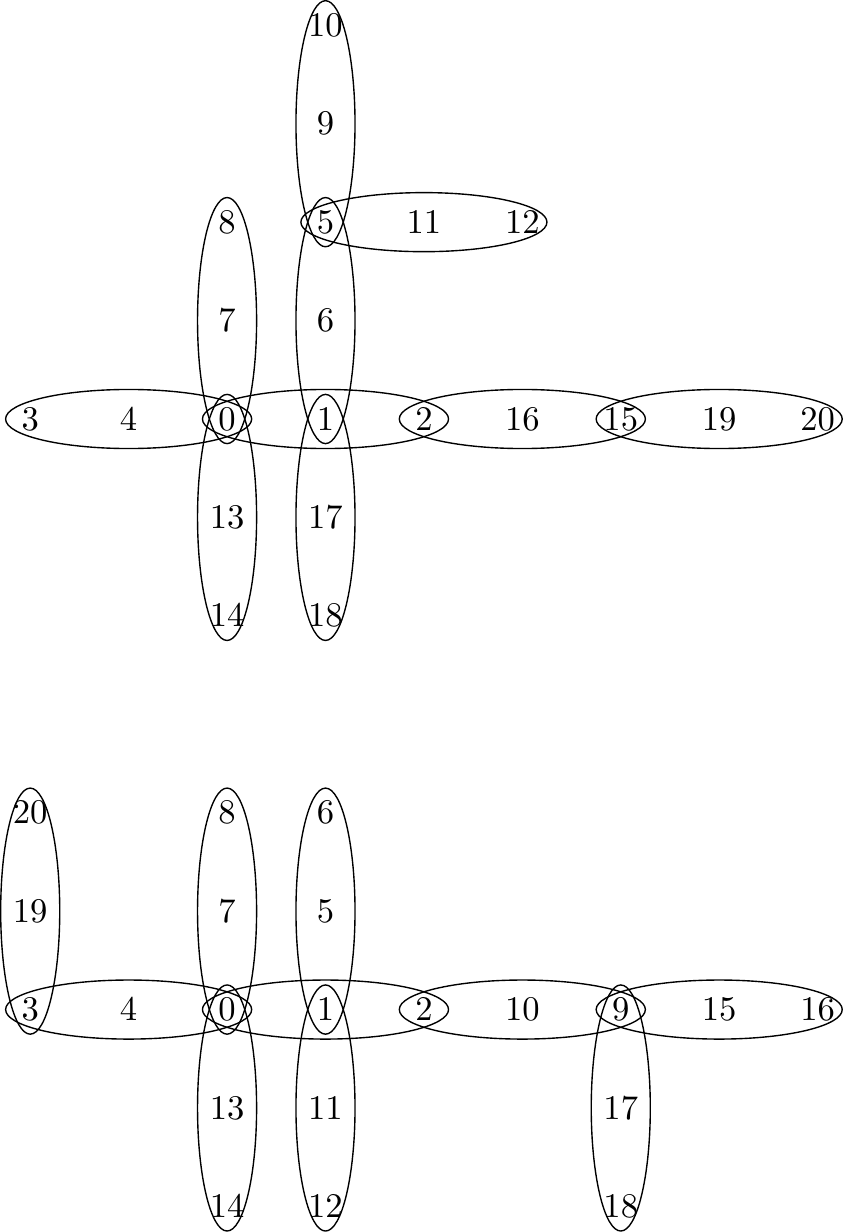}
\caption{The hypertree $H_3$ above and the hypertree $H_4$ below which are not isomorhpic but have the same chromatic symmetric function.}
\label{fig:H3H4}
\end{figure}

\section{More types of coloring}
\label{sec:other}
In this section we show that plurigraph coloring can be used to model oriented coloring and acyclic coloring.
Moreover, the plurigraphs used to model these types of coloring fall outside the realm of those used to model hypergraph coloring.
This shows that plurigraphs provide a new and uniform framework for studying many types of coloring problems and their associated nc-symmetric functions.

\subsection{Oriented coloring}
An \emph{oriented graph} $\vec{G} = (V, \vec{E})$ is an orientation on some simple graph $G = (V,E)$.
This means an oriented graph is a directed graph with no loops or opposite arcs.
A \emph{tournament} is an orientation of a complete graph.
A proper coloring of an oriented graph $\vec{G} = (V, \vec{E})$ is a map $f: V \to \P$ such that:
\begin{itemize}
\item If $(u,v) \in \vec{E}$, then $f(u) \neq f(v)$.
\item If $(u_1, v_1), (u_2, v_2) \in \vec{E}$, then $f(u_1) \neq f(v_2)$ or $f(u_2) \neq f(v_1)$.
\end{itemize}
Notice if $f: V \to \P$ is a proper coloring using $k$ colors, then $f$ will induce a homomorphism from $\vec{G}$ to a tournament $\vec{K}_k$.
When viewed in this sense, we see how oriented coloring is an oriented analog of usual graph coloring.
Oriented coloring was first defined by Courcelle~\cite{C94}.
For a survey of oriented coloring see~\cite{S01,S16}.

Oriented coloring is a graph-like scheduling problem.
For an oriented graph $\vec{G} = (V, \vec{E})$ we let $S_{\vec{G}}$ denote the scheduling problem corresponding to giving a proper oriented coloring of the vertices of $\vec{G}$.
We then have
$$S_{\vec{G}} = \left( \bigwedge_{(u,v) \in \vec{E}} (x_u \neq x_v) \right) \wedge \left(\bigwedge_{((u,v), (u',v')) \in \vec{E} \times \vec{E}} \left( (x_u \neq x_{v'}) \vee (x_{u'} \neq x_v) \right) \right).$$
Also, we let $\G_{\vec{G}}$ denote the plurigraph corresponding to $S_{\vec{G}}$, and so $\G_{\vec{G}} = (V, \E_{\vec{G}})$ where
$$\E_{\vec{G}} = \{(V,\{uv\}) : (u,v) \in \vec{E}\} \cup \{(V, \{uv', u'v\}) : ((u,v), (u',v')) \in \vec{E} \times \vec{E}\}.$$

\begin{ex}
Let $\vec{G} = ([4], \{(1,2), (3,4)\})$, then 
$$S_{\vec{G}} = (x_1 \neq x_2) \wedge (x_3 \neq x_4) \wedge ((x_1 \neq x_4) \vee (x_3 \neq x_2)).$$
Giving a proper coloring of the oriented graph $\vec{G}$ is equivalent to giving a proper coloring of the plurigraph 
$$\G _{\vec{G}}= ([4], \{([4], \{12\}), ([4], \{34\}), ([4], \{14, 23\})\}).$$
Visual representations of $\vec{G}$ and $\G_{\vec{G}}$ can be seen in Figure~\ref{fig:oriented}.
\label{ex:oriented}
\end{ex}

Notice that Example~\ref{ex:oriented} shows that oriented graph coloring in general cannot be modeled with graph coloring or hypergraph coloring.
The pluriedge $([4], \{14, 23\})$ has two connected components which are not singletons.
For the plurigraph corresponding to a hypergraph as in Example~\ref{ex:graph} each pluriedge has a unique connected component which is not a singleton.

\begin{figure}
\centering
\includegraphics[scale = 0.8]{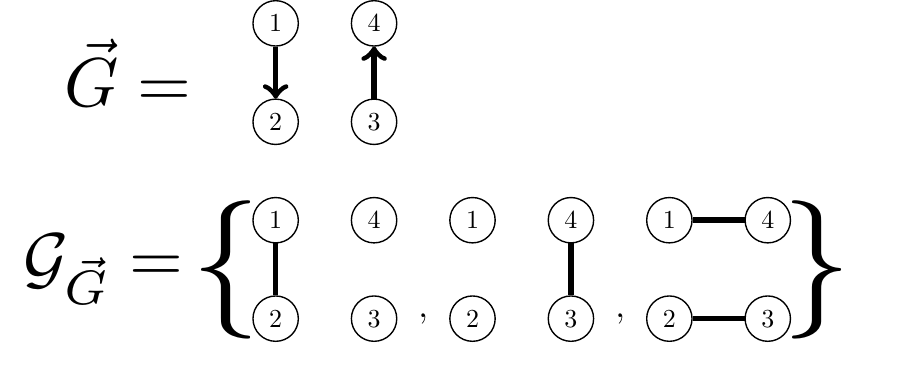}
\caption{A visual representation of $\vec{G}$ and $\G_{\vec{G}}$ from Example~\ref{ex:oriented}}
\label{fig:oriented}
\end{figure}

\subsection{Acyclic coloring}
Given a graph $G = (V,E)$ an \emph{acyclic coloring} of $G$ is $f:V \to \P$ such that:
\begin{itemize}
\item The map $f$ is a proper coloring of $G$.
\item Each cycle in $G$ uses at least three colors.
\end{itemize}
Acyclic coloring was introduced by Gr{\"u}nbaum~\cite{Gr73}.
For an overview of results on acyclic coloring and a discussion of applications of acyclic coloring we refer the reader to~\cite[Section 9]{plane}.
Notice that $f:V \to \P$ is an acyclic coloring if and only if $f$ is a proper coloring and $f$ uses at least three colors on every even length cycle of $G$.
Let $\C(G)$ denote the collection of all even length cycles of $G$.
For any $C \in \C(G)$ let $\len C$ denote the length of the cycle, and let $C_i$ denote the $i$th vertex of $C$ so edges of $C$ are of the form $C_i C_{i+1}$ with indices taken module $\len C$.
Thus, we have the following graph-like scheduling problem $S_{a,G}$ defined by the formula below
$$\bigwedge_{uv \in E} (x_u \neq x_v) \wedge \bigwedge_{C \in \C(G)} \left(\bigvee_{0 \leq i < j< \frac{\len C}{2}} (x_{C_{2i+1}} \neq x_{C_{2j+1}}) \vee \bigvee_{0 < i < j \leq \frac{\len C}{2}} (x_{C_{2i}} \neq x_{C_{2j}})\right) $$
which corresponds to giving an acyclic coloring of $G$.

\begin{ex}
Let $G = ([4], \{12,23,34,14\})$ be the $4$-cycle, then
$$S_{a,G} = (x_1 \neq x_2) \wedge (x_2 \neq x_3) \wedge (x_3 \neq x_4) \wedge (x_1 \neq x_4) \wedge ((x_1 \neq x_3) \vee (x_2 \neq x_4)).$$
This graph-like scheduling problem corresponds to coloring the plurigraph $\G_{a,G}$ where $G$ and $\G_{a,G}$ are shown in Figure~\ref{fig:acyclic}.
Again we see this coloring problem cannot be realized using hypergraphs.
\label{ex:acyclic}
\end{ex}

\begin{figure}
\centering
\includegraphics[scale = 0.8]{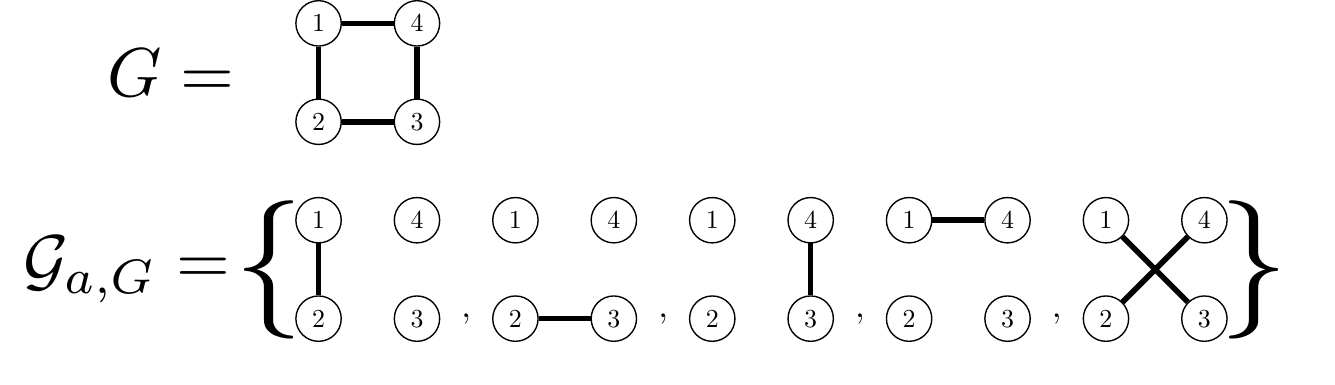}
\caption{The graph $G$ and plurigraph $\G_{a,G}$ from Example~\ref{ex:acyclic}}
\label{fig:acyclic}
\end{figure}

One equivalent characterization of acyclic coloring is that it is a map $f: V \to \P$ such that $f$ is a proper coloring and any subgraph of $G$ induced by any two color classes of $f$ must be a forest.
A \emph{star} is any complete bipartite graph $K_{1,n}$ for $n \geq 0$.
That is, a star is a tree on $n +1$ vertices with a vertex of degree $n$.
A \emph{star forest} is a forest in which all components of stars.
One strengthening of an acyclic coloring known as a \emph{star coloring} was also introduced by Gr{\"u}nbaum~\cite{Gr73}. A star coloring is a map $f: V \to \P$ such that $f$ is a proper coloring and any subgraph of $G$ induced by any two color classes of $f$ must be a star forest.
We see star coloring is also an instance of plurigraph coloring.

A map $f:V \to \P$ is a star coloring if it is a proper coloring and has no bipartite $P_4$, where $P_4$ is the path on $4$ vertices.
Let $\mathcal{P}_4(G)$ denote the set of paths on 4 vertices in $G$.
For any $P \in \mathcal{P}_4(G)$ let $P_i$ denote the $i$th vertex of $P$ so that $P$ has edge set $\{P_1P_2, P_2P_3, P_3P_4\}$.
We have the following graph-like scheduling problem
$$S_{s,G} := \left(\bigwedge_{uv \in E} (x_u \neq x_v)\right) \wedge \left(\bigwedge_{P \in \mathcal{P}_4(G)} \left((x_{P_1} \neq x_{P_3}) \vee (x_{P_2} \neq x_{P_4})\right) \right)$$
corresponding to giving a star coloring of $G$.

\section*{Acknowledgments}

The author wishes to thank the anonymous reviewer for his or her reading and helpful suggestions.

\appendix
\section*{Appendix: SageMath code}
\label{app}

In this appendix we include the SageMath code used to compute the chromatic symmetric functions of the hypertrees given in Section~\ref{sec:dist}.
At this time the chromatic symmetric function is implemented for graphs in SageMath, but is not yet implemented for hypergraphs.
So, below we include our code.
The function \texttt{CSF(V,E)} returns the chromatic symmetric function of a hypergraph with vertex set \texttt{V} and hyperedge set \texttt{E}.
It uses the powersum expansion of the chromatic symmetric function.
We have also included the functions which find the connected components of a hypergraph.
The hypertrees \texttt{H1}, \texttt{H2}, \texttt{H3}, and \texttt{H4} are the same hypertrees from Section~\ref{sec:dist}.

\begin{verbatim}
p = SymmetricFunctions(QQ).power()

def find_component(E,v):
# Find component containing the vertex v given hyperedge set E
    component = [v]
    Q = [v]
    while (len(Q)>0):
        u = Q.pop()
        for e in E:
            if u in e:
                for w in e:
                    if w not in component:
                        component.append(w)
                        Q.append(w)
    return component

def find_components(V,E):
# Find all components of the hypergraph (V,E)
    done = []
    components = []
    for v in V:
        if v not in done:
            c = find_component(E,v)
            done.extend(c)
            components.append(c)
    return components

def find_components_partition(V,E):
# Find the integer partition for the hypergraph (V,E)
    components = find_components(V,E)
    partition = map(len, components)
    partition.sort(reverse=True)
    return partition

def CSF(V,E):
    # Find CSF of the hypergraph (V,E)
    X = 0
    for A in subsets(E):
        X = X + (-1)^(len(A))*p(find_components_partition(V,A))
    return X

V = range(21)

E1 = ((0,1,2),(0,3,4),(1,5,6),(0,7,8),(2,9,10),(1,11,12),(9,13,14),
      (16,3,15),(17,18,7),(19,20,13))

E2 = ((0,1,2),(0,3,4),(1,5,6),(0,7,8),(2,9,10),(1,11,12),(9,13,14),
      (16,3,15),(17,18,5),(19,20,15))

E3 = ((0,1,2),(0,3,4),(1,5,6),(0,7,8),(5,9,10),(5,11,12),(0,13,14),
      (16,2,15),(1,17,18),(19,20,15))

E4 = ((0,1,2),(0,3,4),(1,5,6),(0,7,8),(2,9,10),(1,11,12),(0,13,14),
      (16,9,15),(17,18,9),(3,19,20))

H1 = Hypergraph(E1)
H2 = Hypergraph(E2)
H3 = Hypergraph(E3)
H4 = Hypergraph(E4)

H1.is_isomorphic(H2)
CSF(V,E1) == CSF(V,E2)

H3.is_isomorphic(H4)
CSF(V,E3) == CSF(V,E4)

\end{verbatim}
\end{document}